\documentclass[a4paper]{amsart}
\usepackage{amsmath,amsthm}
\usepackage{amsfonts}
\usepackage{latexsym}
\usepackage{amsxtra}

\numberwithin{equation}{section}

\newtheorem{theorem}{Theorem}[section]
\newtheorem{corollary}[theorem]{Corollary}
\newtheorem{lemma}[theorem]{Lemma}
\newtheorem{example}[theorem]{Example}

\newtheorem{remark}[theorem]{Remark}

\DeclareMathOperator{\idx}{\mathrm{i}}
\DeclareMathOperator{\Ker}{\mathrm{ker}}
\DeclareMathOperator{\Fr}{\mathrm{Fr}\,}
\newcommand{\cl}[1]{\overline{#1}}
\newcommand{\R}{\mathbb{R}}

\newcommand{\Z}{\mathbb{Z}}
\newcommand{\F}{\mathcal{F}}
\newcommand{\Oc}{\mathcal{O}}
\newcommand{\sign}{\mathop\mathrm{sign}\nolimits}
\newcommand{\e}{\varepsilon}
\newcommand{\s}{\mathfrak{s}}
\renewcommand{\t}[1]{\widetilde{#1}}

\usepackage[usenames]{color}

\begin{document}

\author{Alessandro Calamai}
\address[Alessandro Calamai]{Dipartimento di Scienze Matematiche, Universit\`a 
Politecnica delle Marche, Via Brecce Bianche, 60131 Ancona, Italy}
\author{Marco Spadini}
\address[Marco Spadini]{Dipartimento di Matematica Applicata, Universit\`a di 
Firenze, Via Santa Marta 3, 50139 Firenze, Italy}
\title[Forced oscillations for a class of constrained ODEs]{Branches of forced 
oscillations for a class of constrained ODEs: a topological approach}

\begin{abstract}
We apply topological methods to obtain global continuation results for
harmonic solutions of some periodically perturbed ordinary
differential equations on a $k$-dimensional differentiable manifold
$M \subseteq \R^m$. 
We assume that $M$ is globally defined as the zero set of a smooth
map and, as a first step, 
we determine a formula which reduces the computation
of the degree of a tangent vector field on $M$ to the Brouwer degree
of a suitable map in $\R^m$.
As further applications, we study the set of harmonic solutions to
periodic semi-explicit differential-algebraic equations.
\end{abstract}


\maketitle

\section{Introduction and preliminaries} 

In this paper we study $T$-periodic solutions of some parametrized families of 
$T$-periodic constrained ordinary differential equations (ODEs). More precisely, 
we study periodically perturbed autonomous ODEs on a differentiable submanifold 
of some Euclidean space, under the assumption that such a manifold is globally 
defined as the zero set of a smooth map. We consider the two different cases
of nontrivial unperturbed equation and of perturbation of the zero vector field.
We adopt a topological approach and we make use of results which are based on the 
fixed point index. However, our techniques require just the notion of the degree 
(often called characteristic or rotation) of tangent vector fields to differentiable 
manifolds, which in the `flat' case, that is when the manifold is an open subset
of an Euclidean space, is essentially the well known Brouwer degree. As an application 
of our results, we study $T$-periodic solutions of particular parametrized 
differential-algebraic equations (DAEs), for which we will prove global 
continuation results.

Recently, differential-algebraic equations have received an increasing interest 
due, in particular, to applications in engineering and have been the subject of
extensive study (see e.g.\ \cite{KM} for a comprehensive treatment) aimed mostly 
(but not only) to numerical methods. Our approach here, inspired by \cite{Ca10} 
and \cite{Spa}, is directed towards qualitative theory of some particular DAEs 
which are studied by means of topological methods, making use of the equivalence 
of the given equations and suitable ODEs on manifolds. Relatively to \cite{Ca10,Spa},
here we operate a change of perspective: assuming the viewpoint of ODEs on manifold 
allows us to present the matter in a general and extensively studied framework
(see e.g.\ \cite{FPS00}).
\smallskip

Our first goal is to obtain a formula for the computation of the degree of tangent 
fields to a $k$-dimensional differentiable submanifold $M$ of $\R^m$,
in the particular case when the manifold is defined 
implicitly as the zero set of a smooth function $g:U\to\R^s$, with
$s=m-k$ and $U\subseteq\R^m$ 
open and connected, and assuming that with an appropriate choice of (orthonormal) 
coordinates one can decompose $\R^m$ as $\R^k\times\R^s$ in such a way that the 
Jacobian matrix of $g$ with respect to the last $k$ variables, $\partial_2g(x,y)$, 
is nonsingular for all $(x,y)\in U$.
Notice that, in this case, $0\in\R^s$ is a regular value of $g$ so
that $M=g^{-1}(0)$ is a smooth submanifold  of $\R^m=\R^k\times\R^s$
of codimension $s$.

The formula we find (see Theorem \ref{formuladeg} below) reduces the computation of the 
degree of a tangent vector field on $M$ to that of an appropriate map in $\R^m$.  More 
precisely, let $\varphi:M\to\R^m$ be tangent to $M$, in the sense that $\varphi(\xi)$ 
belongs to the tangent space $T_{\xi}M$ of $M$ at $\xi$ for any $\xi\in M$. Let also 
$\t \varphi$ be \emph{any} extension of $\varphi$ to $U$. With a small abuse of notation, 
we will write, according to the above decomposition,
\[
 \t\varphi(\xi)=\t\varphi(x,y)=\big(\t\varphi_1(x,y),\t\varphi_2(x,y)\big),
\]
and define $\F:U\to\R^m$ as $\F(x,y)=\big(\t\varphi_1(x,y),g(x,y)\big)$, for
any $(x,y)\in U$. We will prove that 
\begin{equation}\label{ffond}
 \deg(\varphi,M)=\s \deg(\F,U)
\end{equation}
where $\s$ is the (constant) sign of $\det(\partial_2g)$ on the connected set $U$.
Observe that the just defined vector field $\F$ on $U$ may well not be tangent to
$M$. In fact, on $M$, the second component of $\F$ is forced to be zero regardless of
the shape of $M$.
 
The above formula \eqref{ffond} is equivalent to 
a result proved in \cite{Spa} but we provide here a simplified proof. Notice that \eqref{ffond} 
does not depend on the chosen extension of $\varphi$.
Notice also that, since in Euclidean spaces vector fields can be regarded as maps and 
vice versa,  the degree of the vector field $\F$ that appears in the second member of
\eqref{ffond} is essentially the well known Brouwer degree, with respect to $0$, of 
$\F$ seen as a map. Hence the degree of $\F$, having a simpler nature than that of $\varphi$, 
is `morally' easier to compute.

\medskip

As an application we study the set of harmonic solutions of the
following pa\-ram\-etrized differential 
equations on a manifold $M\subseteq\R^m$, with $M=g^{-1}(0)$ and $g$ as above:
\begin{subequations}\label{unoedue}
\begin{equation}
\dot \xi=f(\xi)+\lambda h(t,\xi),\quad\lambda\geq 0  \label{uno}
\end{equation}
and
\begin{equation}\label{due}
\dot \xi=\lambda h(t,\xi),\quad\lambda\geq 0
\end{equation}
\end{subequations}
where $h:\R\times M\to\R^m$ and $f:M\to\R^m$ are continuous maps with the property that 
$f(\xi)$ and $h(t,\xi)$ belong to $T_\xi M$ for any $(t,\xi)\in\R\times M$, and $h$ is 
$T$-periodic in the first variable.

Notice that locally $M$ can be represented as graph of some map from an open subset 
of $\R^k$ to $\R^s$, with $k=m-s$. Thus equations \eqref{unoedue} can be locally simplified. In 
view of this fact one might think that it is possible to reduce equations \eqref{unoedue}
to ordinary differential equations in $\R^k$. It is not so. In fact, globally, 
$M$ may not be the graph of a map from an open subset of $\R^k$ to $\R^s$ as, for 
instance, when $U=\R^3$ and $g:\R\times\R^2\to\R^2$ is given by
\[
g(x,y)=g(x;y_1,y_2)=\big(e^{y_1}\cos y_2-x, e^{y_1}\sin y_2-x\big).
\]
In this case, although $\det\partial_2g(x,y)\neq 0$, one clearly has that the 
$1$-dimensional manifold $M=g^{-1}(0)$ is not the graph of a function 
$x\mapsto\big(y_1(x),y_2(x)\big)$. In fact, $M$ consists of infinitely many 
connected components each lying in a plane $y_2=\frac{\pi}{4}+\ell\pi$ for $\ell\in\Z$.
(See also Examples \ref{elicafinta} and \ref{elicavera} below.)

Observe also that even when $M$ is a (global) graph of some map $\Gamma$, the expression 
of $\Gamma$ might be too complicated or impossible to determine analytically, so that the 
decoupled versions of equations \eqref{unoedue} may be too difficult to use. A simple 
example of this fact is obtained by taking $k=s=1$, $U=\R\times\R$ and $g(x,y)=y^7+y-x^2+x^5$.
\medskip

As further applications, we will deduce the results of \cite{Ca10,Spa} about harmonic 
solutions of periodic semi-explicit differential-algebraic equations that have either 
the form 
\begin{subequations}\label{DAEs}
\begin{equation}\label{tipo1}
 \left\{
\begin{array}{l}
 \dot x=\gamma(x,y)+\lambda \sigma(t,x,y),\quad\lambda\geq 0,\\
 g(x,y)=0,
\end{array}
\right.
\end{equation}
or 
\begin{equation}\label{tipo2}
 \left\{
\begin{array}{l}
 \dot x=\lambda\sigma(t,x,y),\quad\lambda\geq 0,\\
 g(x,y)=0.
\end{array}
\right.
\end{equation} 
\end{subequations}
Here $U\subseteq\R^k\times\R^s$ is a connected open set, $g:U\to\R^s$
is as above, 
$\gamma:U\to\R^k$ and $\sigma:\R\times U\to\R^k$ are continuous
maps, and $\sigma$ is $T$-periodic in $t$ for a given $T>0$. In fact,
as we shall see, equations \eqref{tipo1} 
and \eqref{tipo2} are equivalent to \eqref{uno} and \eqref{due}, respectively, for 
appropriate vector fields $f$ and $h$ on the manifold $g^{-1}(0)$.
Notice that, as remarked above, although the set $g^{-1}(0)$ is locally
the graph of some map of an open set of $\R^k$ to $\R^s$ so that equations
\eqref{tipo1} and \eqref{tipo2} can be locally decoupled, it is not always possible
or convenient to do globally so.

\section{Tangent vector fields and the notion of degree}

We now recall some basic notions about tangent vector fields on manifolds, and introduce the notion 
of degree of an admissible tangent vector field.

Let $M\subseteq\R^m$ be a manifold. Let $w$ be a tangent vector field on $M$, that 
is, a continuous map $w:M\to\R^m$ with the property that $w(\zeta)\in T_{\zeta}M$ for any 
$\zeta\in M$. 
If $w$ is (Fr\'echet) differentiable at $\zeta\in M$ and $w(\zeta)=0$, then the differential 
$dw_{\zeta} : T_{\zeta}M\to\R^m$ maps $T_{\zeta}M$ into itself (see e.g.\ \cite{Mi}), so that 
the determinant $\det dw_{\zeta}$ of $dw_{\zeta}$ is defined. If, in addition, $\zeta$ is a 
nondegenerate zero (i.e.\ $dw_{\zeta} : T_{\zeta}M\to\R^m$ is injective) then $\zeta$ is an 
isolated zero and $\det dw_{\zeta}\neq 0$. 

Let $W$ be an open subset of $M$ in which we assume $w$ \emph{admissible} (for the degree); that 
is, the set $w^{-1}(0) \cap W$ is compact. Then, one can associate to the pair $(w,W)$ an integer, 
$\deg(w,W)$, called the \emph{degree (or characteristic) of the vector field $w$ in $W$}, which, 
in a sense, counts (algebraically) the zeros of $w$ in $W$ (see e.g.\ \cite{FPS05, H, Mi} and 
references therein). In fact, when the zeros of $w$ are all nondegenerate, then the set 
$w^{-1}(0)\cap W$ is finite and
\begin{equation}\label{sommasegni}
\deg(w,W)=\sum_{\zeta\in w^{-1}(0)\cap W}{\rm sign} \det dw_{\zeta}.
\end{equation}
Observe that in the flat case, i.e.\ when $M = \R^m$, $\deg(w,W)$ is just the classical
Brouwer degree with respect to zero where $V$ is any
bounded open neighborhood of $w^{-1}(0) \cap W$ whose closure is
contained in $W$.

The notion of degree of an admissible tangent vector field plays a crucial role throughout
this paper. It enjoys a number of properties some of which we report here for the sake of future 
reference.

\begin{description}
\item[Additivity]
{\em Let $w$ be admissible in $W$. If $W_1$ and $W_2$ are two disjoint open subsets of $W$ whose 
union contains $w^{-1}(0)\cap W$, then}
\[
\deg(w,W) = \deg(w,W_1)+\deg(w,W_2).
\]
\item[Homotopy Invariance]
{\em Let $h:M\times[0,1]\to\R^m$ be an admissible homotopy (of tangent vector fields) in $W$; that 
is, $h(\zeta,\lambda)\in T_{\zeta}M$ for all $(\zeta,\lambda)\in M\times[0,1]$ and 
$h^{-1}(0)\cap W\times[0,1]$ is compact. Then $\deg\big(h(\cdot,\lambda),W\big)$ is independent of 
$\lambda$.}
\item[Solution]
{\em If $w$ is admissible in $W$ and $\deg(w,W)\neq 0$, then $w$ has a zero in $W$.}
\end{description}
The Additivity Property implies the following important property:
\begin{description}
 \item[Excision]{\em Let $(w,W)$ be admissible. If $V\subseteq W$ is open and contains 
$w^{-1}(0)\cap W$, then $\deg(w,W) = \deg(w,V)$.} 
\end{description}
Another property that plays an important role in this paper is the following one which allows
the comparison between the degrees of vector fields that correspond under diffeomorphisms. Recall 
that if $v:N\to\R^n$ and $w:M\to\R^m$ are tangent vector fields on the differentiable manifolds
$N\subseteq\R^n$ and $M\subseteq\R^m$, and if $\rho:W\to V$ is a diffeomorphism from an open subset 
$W$ of $M$ onto an open subset $V$ of $N$, we say that $v|_V$ and $w|_W$ \emph{correspond under 
$\rho$} when $v(\rho(\zeta))=d\rho_{\zeta}\big(w(\zeta)\big)$ for all $\zeta\in W$.

\smallskip
\begin{description}
\item[Invariance under diffeomorphisms] {\em Let $M\subseteq\R^n$ and $N\subseteq\R^m$ 
be differentiable manifolds and let $v:N\to\R^n$ and $w:M\to\R^m$ be tangent vector fields.
Assume that $v|_V$ and $w|_W$ correspond under some diffeomorphism. Then, if either $v$ is 
admissible in $V$ or or $w$ is admissible in $W$, so is the other and 
\[
\deg(v,V)=\deg(w,W).
\]}
\end{description}

\begin{remark}\label{sper}
Let $W\subseteq M$ be open and 
relatively compact. If $w:M\to\R^m$ is such that $w(\zeta)\neq 0$ on the boundary $\Fr(W)$ 
of $W$, then $(w,W)$ is admissible. Let $\e=\min_{\zeta\in\Fr(W)}|w(\zeta)|$. Then, for any 
$v:M\to\R^m$ such that $\max_{\zeta\in\Fr(W)}|w(\zeta)-v(\zeta)|<\e$, we have that $(v,W)$ is 
admissible and that the homotopy $h:M\times[0,1]\to\R^m$ given by 
\[
h(\zeta,\lambda)=\lambda w(\zeta)+(1-\lambda) v(\zeta)
\]
is admissible in $W$. Hence, by the Homotopy Invariance Property, 
\[
\deg(w,W)=\deg(v,W).
\]
\end{remark}

The Excision Property allows the introduction of the notion of index of an isolated zero 
of a tangent vector field. Let $\zeta\in M$ be an isolated zero of $w$. Clearly, $\deg(w,V)$ 
is well defined for each open $V\subseteq M$ such that $V\cap w^{-1}(0)=\{ \zeta\}$. By the 
Excision Property $\deg(w,V)$ is constant with respect to such $V$'s. This common value of 
$\deg(w,V)$ is, by definition, the \emph{index of $w$ at $ \zeta$}, and is denoted by 
$\mathrm{i}\,(w, \zeta)$. Using this notation, if $(w,W)$ is admissible, by the Additivity 
Property we get that if all the zeros in $W$ of $w$ are isolated, then
\begin{equation}\label{sommaindici}
\deg(w,W)=\sum_{\zeta\in w^{-1}(0)\cap W} \mathrm{i}\,(w,\zeta).
\end{equation}
By formula \eqref{sommasegni} we have that if $\zeta$ is a nondegenerate zero of $w$, then
\begin{equation}\label{idx_loc}
 \mathrm{i}\,(w,\zeta)=\sign\det d w_{\zeta}.
\end{equation}
Notice that \eqref{sommasegni} and \eqref{sommaindici} differ in the fact that, in the
latter, the zeros of $w$ are not necessarily nondegenerate as they have to be in the former. 
In fact, in \eqref{sommaindici}, $w$ need not be differentiable at its zeros.

\section{Tangent vector fields on implicitly defined manifolds}\label{sectan}

Let $\Psi:\R\times M\to\R^m$ be a (time-dependent) tangent vector field on $M\subseteq\R^m$, 
that is a continuous map with the property that $\Psi(t,\zeta)\in T_{\zeta}M$ for each 
$(t,\zeta)\in\R\times M$. 
Assume that there is
a connected open subset $U$ of $\R^m$ and a smooth map $g:U\to\R^s$ with the 
property that $M=g^{-1}(0)$.
Suppose that with an orthogonal transformation, if necessary, one can write
$\R^m=\R^k\times\R^s$, in such a way that the partial derivative of $g$ 
with respect to the second variable, $\partial_2 g(x,y)$, is invertible for 
each $(x,y)\in U$.
By this we mean that there exists an orthogonal transformation $P$ of
$\R^m$ such that the above property holds with the map
$\tilde g= g\circ P : P^{-1}(U) \to \R^s$ in place of $g$, and with
the open set $P^{-1}(U)$ replacing $U$. 

To illustrate this point consider by way of example $m=3$, $s=1$, $U=\R^3$,
and $g(\xi_1, \xi_2, \xi_3)=\xi_2+\xi_1\xi_3$
(here $g^{-1}(0)$ is a hyperbolic paraboloid).
Setting $(x_1,x_2,y)=(\xi_1, \xi_3, \xi_2)$ and 
$\tilde g(x_1,x_2,y)=y+x_1x_2$, one has that the partial derivative of
$\tilde g$ with respect to $y$ is $1$.

As previously remarked in the Introduction, the restrictive assumption
we impose on $g$ does not imply that $M=g^{-1}(0)$ is globally a
graph. Likewise, one should note that in general a manifold may not be
representable as the zero set of a function for which the above
condition holds, as in the case of the sphere in $\R^m$.

According to the above decomposition of $\R^m$ we can write, for any
$\xi\in M$, $\xi=(x,y)$ and, for any $t\in\R$ 
\[
 \Psi(t,\xi)=\Psi(t,x,y)=\big(\Psi_1(t,x,y),
                              \Psi_2(t,x,y)\big).
\]
Notice that one must have 
\begin{equation}\label{phi2}
\Psi_2(t,x,y)=-\big(\partial_2g(x,y)\big)^{-1}\partial_1g(x,y)\Psi_1(t,x,y).
\end{equation}
In fact, $\Psi(t,\xi)\in T_{\xi}M$ being equivalent to $\Psi(t,\xi)\in\Ker g'(x,y)$,
one has for each $(t,x,y)\in\R\times M$ that
\[
0= g'(x,y)\Psi(t,x,y)=
   \partial_1g(x,y)\Psi_1(t,x,y)+\partial_2g(x,y)\Psi_2(t,x,y),
\]
which implies \eqref{phi2}; here $g'(x,y)$ denotes the Fr\'echet derivative of $g$ at 
$(x,y)$.

\medskip

We now focus on differential equations on $M$ and write them
equivalently (in a sense specified below) as differential-algebraic
equations.

Let us consider the following differential equation on $M$:
\begin{equation}\label{eq.phi.onM}
 \dot\xi= \Psi(t,\xi).
\end{equation}
By a \emph{solution} of \eqref{eq.phi.onM} we mean a $C^1$ curve
$\xi:J \to \R^m$, defined on a (nontrivial) interval $J\subseteq \R$,
which satisfies the conditions $\xi(t) \in M$ and $\dot\xi(t) =
\Psi(t,\xi(t))$, identically on $J$.
                              ̇
We need the following fact.

\begin{remark}\label{remext0}
If $\Psi$ is as above, since $\R\times M$ is a closed subset of the metric space $\R\times U$, the well known Tietze's
Theorem (see e.g.\ \cite{Dug}) implies that there exists an extension
$\t\Psi:\R\times U\to\R^m$ 
of $\Psi$. 
\end{remark}

Consider also the `extended' equation on the neighborhood $U$ of $M$
in $\R^m$:
\begin{equation}\label{eq.phitilde.onM}
 \dot\eta= \t\Psi(t,\eta),
\end{equation}
where $\t\Psi$ is any extension of $\Psi$ as in Remark \ref{remext0}.
Observe that the solutions of \eqref{eq.phi.onM} are also solutions of
\eqref{eq.phitilde.onM};
conversely, the solutions of \eqref{eq.phitilde.onM} that meet $M$ do
actually lie on $M$ and thus are solutions of \eqref{eq.phi.onM}.

Equation \eqref{eq.phitilde.onM} can be conveniently written, setting
$\eta=(x,y)$, as the following system:
\begin{equation}\label{eq.phi12.onM}
\left\{
\begin{array}{l}
 \dot x= \Psi_1(t,x,y),\\
 \dot y= \Psi_2(t,x,y)
\end{array}\right.
\end{equation}
where, for the sake of simplicity, $\t\Psi$ has been replaced by $\Psi$.

We claim that \eqref{eq.phi.onM} is equivalent to the following
differential-algebraic equation:
\begin{equation}\label{eq.phi.DAE}
\left\{
\begin{array}{l}
 \dot x= \Psi_1(t,x,y),\\
 g(x,y)=0.
\end{array}\right.
\end{equation}
Here by a \emph{solution} of \eqref{eq.phi.DAE} we mean a pair of
$C^1$ functions $x:J \to \R^k$ and $y:J \to \R^s$, $J$ an interval,
with the property that $\dot x(t)=\Psi_1(t,x(t),y(t))$ and
$g\big(x(t),y(t)\big)=0$ for all $t\in J$.

To prove the claim, let $x:J\to\R^k$ and $y:J\to\R^s$ be $C^1$ maps
defined on an interval $J$ 
with the property that $t\mapsto\xi(t)=\big(x(t),y(t)\big)$ is a
solution of \eqref{eq.phi.onM}. Then, for all $t\in J$, $\dot x(t)=\Psi_1(t,x(t),y(t))$
and, since $\big(x(t),y(t)\big)\in M$, we have $g\big(x(t),y(t)\big)=0$. 
Conversely, if $t\mapsto\big(x(t),y(t)\big)$ is a solution of 
\eqref{eq.phi.DAE} then, differentiating $g\big(x(t),y(t)\big)=0$ at any $t\in J$, one gets
\[
 \partial_1g\big(x(t),y(t)\big)\dot x(t)+\partial_2g\big(x(t),y(t)\big)\dot y(t)=0.
\]
So that 
\begin{multline*}
\dot y(t)=-(\partial_2g\big(x(t),y(t)\big))^{-1}\partial_1g\big(x(t),y(t)\big)\dot x(t)\\
         =-(\partial_2g\big(x(t),y(t)\big))^{-1}\partial_1g\big(x(t),y(t)\big)
                                                           \Psi_1\big(t,x(t),y(t)\big).
\end{multline*}
Taking into account \eqref{phi2} and the fact that the solution meets
$M$, we have the claim.

\section{Computation of the degree}

As in the previous section, let $M \subseteq \R^m$ be a differentiable manifold that is globally 
defined as a zero set of a suitable map $g:U\to\R^s$, $U\subseteq\R^m$. Here we give a formula 
for the degree of tangents vector fields on $M$ in terms of (potentially easier to compute) degree 
of appropriate vector fields on $U$. The main result of this section is Theorem \ref{formuladeg} 
below, which is equivalent to a result of \cite{Spa}. Here we provide a simplified proof.

Throughout this section $\varphi: M \to \R^m$ will be a
continuous tangent vector field on $M$.
As in Remark \ref{remext0}, Tietze's
Theorem implies that there exists an extension $\t\varphi:U\to\R^m$ 
of $\varphi$. 
Thus, it is not restrictive to assume, as we sometimes do, that the given 
tangent vector fields are actually defined on a convenient neighborhood of the manifold $M$. In
fact, although an arbitrary extension of $\varphi$ may have many zeros outside $M$, we are 
interested in the degree of $\varphi$ on $M$ which only takes into account those zeros of 
$\varphi$ that lie on $M$.

\begin{theorem}\label{formuladeg}
Let $U\subseteq\R^k\times\R^s$ be open and connected, let $g:U\to\R^s$ be a smooth function such 
that $\partial_2g(x,y)$ is nonsingular for any $(x,y)\in U$ and let $M=g^{-1}(0)$. Assume that 
$\varphi:M\to\R^k\times\R^s$ is a continuous tangent vector field on $M$, and let $\t\varphi_1$ be the projection on $\R^k$ 
of an arbitrary continuous extension $\t\varphi$ of $\varphi$ to $U$. Define $\F:U\to\R^k\times\R^s$ 
by $\F(x,y)=\big(\t\varphi_1(x,y),g(x,y)\big)$. Then, $\F$ is admissible in $U$ if and only if
so is $\varphi$ in $M$, and
\begin{equation}\label{idgradi}
 \deg(\varphi,M)=\s\deg(\F,U),
\end{equation}
where $\s$ is the constant sign of $\det\partial_2g(x,y)$ for all $(x,y)\in U$.
\end{theorem}

Before we provide the proof of Theorem \ref{formuladeg}, we consider a special case.
Observe that a point $(p,q)\in M$ is a zero of $\varphi$ if and only if it is a zero of $\F$.

\begin{lemma}\label{zerisolati}
Let $U$, $\s$, $\varphi$ be as in Theorem \ref{formuladeg}. Assume that $\varphi$ is $C^1$ and 
let $\t\varphi=\big(\t\varphi_1,\t\varphi_2\big)$ be a $C^1$ extension of $\varphi$ to $U$. Let 
$\F:U\to\R^k\times\R^s$ be given by $\F(x,y)=\big(\t\varphi_1(x,y),g(x,y)\big)$, as in Theorem 
\ref{formuladeg}, and suppose that all the zeros of $\F$ are nondegenerate. Then,
\[
 \deg(\varphi,M)=\mathfrak{s}\deg(\F,U).
\]
\end{lemma}

\begin{proof}
Observe that since the zeros of $\F$ are nondegenerate, they are also isolated. This implies 
that the zeros of $\varphi$ are isolated as well. Let $(p,q)$ be a zero of $\F$. As a first 
step, we will show that
\begin{equation}\label{flocale}
 \idx\big(\varphi,(p,q)\big)=\mathfrak{s}\sign\det d\F_{(p,q)}.
\end{equation} 

 Since $\det\partial_2 g(p,q)\neq 0$, the so-called generalized Gauss algorithm (see e.g.\ 
\cite{G}) yields 
\begin{equation}
\begin{split}\label{fsgnsch}
 \det & d \F_{(p,q)} = \det
\begin{pmatrix}
 \partial_1 \t\varphi_1(p,q) & \partial_2 \t\varphi_1(p,q)\\
\partial_1 g(p,q) & \partial_2 g(p,q)
\end{pmatrix}=\\
&\quad = \det\partial_2 g(p,q)
     \cdot\det \Big(\partial_1\t\varphi_1(p,q)-\partial_2\t\varphi_1(p,q)
                     \big(\partial_2 g(p,q)\big)^{-1}\partial_1 g(p,q)\Big).
\end{split}
\end{equation}
As remarked above,
since $(p,q)$ is a nondegenerate zero of $\F$ then it is also isolated, as a zero, of both
$\F$ and of $\varphi$ on $M$. Let $B=W\times V$, with $W\subseteq\R^k$ and $V\subseteq\R^s$
open, be an isolating neighborhood of $(p,q)$ in $\R^k\times\R^s$ i.e.\ $B$ is such that 
$\F (x,y)\neq(0,0)$ for any $(x,y)\in B\setminus\{(p,q)\}$, and $\varphi(x,y)\neq (0,0)$ 
for any $(x,y)\in B\cap M\setminus\{(p,q)\}$.

Since $\partial_2g(p,q)$ is invertible, the implicit function theorem implies that, taking
a smaller $W$ if necessary, we can assume that there exists a $C^1$ function 
$\gamma:W\to\R^s$ such that $g\big(x,\gamma(x)\big)=0$ for any $x\in W$ and 
$\gamma(W)\subseteq V$. Thus the map $G:x\mapsto\big(x,\gamma(x)\big)$ is a 
diffeomorphism of $W$ onto $B\cap M$ whose inverse is the projection $\pi:B\cap M\to W$ 
given by $\pi(x,y)=x$.

The property of invariance under diffeomorphisms of the degree of tangent vector fields 
implies that
\[
 \deg(\varphi,B\cap M)=\deg\big(\pi\circ\varphi\circ G,W).
\]
Notice that $p$ is an isolated zero of $\pi\circ\varphi\circ G$. Thus, the above relation
becomes
\begin{equation}\label{ididx}
 \idx\big(\varphi,(p,q)\big)=\idx\big(\pi\circ\varphi\circ G,p)
\end{equation}
The differential of $\pi\circ\varphi\circ G$ at $p$ is given by
\[
 \partial_1 \varphi_1(p,q)-\partial_2 \varphi_1(p,q)
\big(\partial_2 g(p,q)\big)^{-1}\partial_1 g(p,q)
\]
(recall that $q=\gamma(p)$), which is equal to
\begin{equation}\label{diftil}
 \partial_1\t\varphi_1(p,q)-\partial_2\t\varphi_1(p,q)
\big(\partial_2 g(p,q)\big)^{-1}\partial_1 g(p,q)
\end{equation}
because the differential of $\varphi$ at $(p,q)$ coincides with the restriction to $T_{(p,q)}M$
of the differential of $\t\varphi$ at the same point. By \eqref{fsgnsch} and the fact that 
$(p,q)$ is a nondegenerate zero of $\F$, it follows that the map in \eqref{diftil} is 
invertible. Therefore, by \eqref{ididx} and \eqref{idx_loc}, we have 
\begin{equation}
\begin{split}\label{idxpsi}
\idx\big(\varphi, & (p,q)\big)=\\
  &=\sign\det \Big(\partial_1\t\varphi_1(p,q)-\partial_2\t\varphi_1(p,q)
                     \big(\partial_2 g(p,q)\big)^{-1}\partial_1 g(p,q)\Big).
\end{split}
\end{equation}
Formula \eqref{flocale} follows from \eqref{fsgnsch} and \eqref{idxpsi}.

To complete the proof, let $(p_1,q_1),\ldots,(p_n,q_n)$ be the zeros of $\F$. Since $\s$ is constant 
on the connected set $U$, from \eqref{sommaindici}, Lemma \ref{zerisolati} and \eqref{sommasegni} 
we have
\begin{align*}
\deg(\varphi, M) & = \sum_{i=1}^n\idx\big(\varphi,(p_i,q_i)\big)=\\ 
&\qquad = \sum_{i=1}^n\mathfrak{s}\sign\det d\F_{(p_i,q_i)} =\mathfrak{s} \deg(\F,U),
\end{align*}
that proves the assertion.
\end{proof}

\begin{proof}[Proof of Theorem \ref{formuladeg}.]
The assertion that $\F$ is admissible in $U$ if and only if so is $\varphi$ in $M$ follows from 
the identity
\[
 \big\{(p,q)\in M:\varphi(p,q)=0\big\} = \big\{(p,q)\in U:\F(p,q)=0\big\},
\]
which can be deduced from the definition of $\F$ and the fact that, according to \eqref{phi2}, 
the projection $\varphi_2(x,y)$ of $\varphi(x,y)=\big(\varphi_1(x,y),\varphi_2(x,y)\big)$, at 
$(x,y)\in M$ onto $\R^s$, is given by
\[
 -\big(\partial_2 g(x,y)\big)^{-1}\partial_1 g(x,y)\varphi_1(x,y).
\]

Assume now that $\F$ is admissible in $U$. 
Let $V$ be an open and bounded subset of $U$ with the property that the closure $\cl{V}$ of $V$ 
is contained in $U$ and that $\F^{-1}(0,0)\subseteq V$. Clearly, $\varphi^{-1}(0,0)\cap M$ is 
contained in $V$ as well and, by the excision property of the degree of a vector field, we get
\[
 \deg(\F,U)=\deg(\F,V),\qquad \deg(\varphi, M)=\deg(\varphi,V\cap M).
\]
Therefore, it is sufficient to prove that 
\begin{equation}\label{sulcomp}
\deg(\varphi,V\cap M)=\s\deg(\F,V).
\end{equation}

We shall deduce equation \eqref{sulcomp} from Lemma \ref{zerisolati} via an approximation 
procedure. 
Given $\e>0$, Sard's Lemma implies that one can find a $C^1$ map $\F^\e:U\to\R^k\times\R^s$, 
$\F^\e=(\F^\e_1,\F^\e_2)$, that has $(0,0)$ as a regular value and such that 
\[
  \max_{(x,y)\in\Fr(V)}\big|\F^\e(x,y)-\F(x,y)\big|<\e.
\]
Define $\t\psi^\e:U\to\R^k\times\R^s$ by 
\[
\t\psi^\e(x,y)=
  \left( \F^\e_1(x,y),-\big(\partial_2 g(x,y)\big)^{-1}\partial_1 g(x,y)\F^\e_1(x,y)\right),
\]
and denote by $\psi^\e$ the restriction of $\t\psi^\e$ to $M$. As in Section 
\ref{sectan}, we see immediately that $\psi^\e$ is a tangent vector
field on $M$.
Recalling formula \eqref{phi2}, one has
\begin{align*}
\sup_{(x,y)\in \Fr (V\cap M)} \big|& \psi^\e(x,y)  -\varphi(x,y)\big| 
  \leq\sup_{(x,y)\in \Fr (V\cap M)}\Big|\F^\e_1(x,y)-\F_1(x,y)\Big| +\\  &\qquad +
    \sup_{(x,y)\in \Fr (V\cap M)}\left|\big(\partial_2 g(x,y)\big)^{-1}\partial_1 g(x,y)
                                  \big(\F^\e_1(x,y)-\F_1(x,y)\big)\right|\\
  & <\e\left(1+\sup_{(x,y)\in\Fr (V\cap M)} 
                      \left\|\big(\partial_2 g(x,y)\big)^{-1}\partial_1 g(x,y)\right\|\right)
\end{align*}
where $|\cdot|$ denotes, according to the space where applied, the Euclidean norm in $\R^k$, 
$\R^s$ or $\R^{k+s}$, and $\|\cdot\|$ denotes the norm of linear operators 
from $\R^k$ to $\R^s$. Thus, by the continuity of the partial derivatives of $g$ and the 
compactness of $\cl V\cap M$, it follows that one can choose $\e$ so small that
\[
 \max_{(x,y)\in\Fr(V)}\big|\F^\e(x,y)-\F(x,y)\big|<\min\{|\F(x,y)|:(x,y)\in\Fr(V)\},
\]
and
\[
 \max_{(x,y)\in\Fr(V\cap M)}\big|\psi^\e(x,y)-\varphi(x,y)\big|
           <\min\{|\varphi(x,y)|:(x,y)\in\Fr(V\cap M)\}.
\]
For such a choice of $\e$ it is easily checked that $\F^\e$ and $\psi^\e$ are admissibly homotopic 
to $\F$ on $V$ and to $\varphi$ on $V\cap M$, respectively (compare Remark \ref{sper}). Thus, 
\begin{equation}\label{appF}
\deg(\F^\e,V)=\deg(\F,V).
\end{equation} 
and 
\begin{equation}\label{apppsi}
 \deg(\psi^\e,V\cap M)=\deg(\varphi,V\cap M)
\end{equation} 

Observe also that because of the assumptions on $g$, any zero of $\F^\e$ is nondegenerate. 
By Lemma \ref{zerisolati} it follows that
\begin{equation}\label{fapp}
 \deg(\psi^\e,V\cap M)=\s\deg(\F^\e,V).
\end{equation}
Now, Equations \eqref{apppsi}, \eqref{fapp} and \eqref{appF} imply \eqref{sulcomp}. This 
completes the proof.
\end{proof}

\begin{example}
 Let $k=s=1$, $U=\R^2$ and $g(x,y)=x^3-y^3-3y$. Consider the tangent vector field on 
$M=g^{-1}(0)$  given by $\varphi(x,y)=\big(x(y^2+1),x^3\big)$. Define $\F:U\to\R^2$
by $\F(x,y)=\big(x(y^2+1),x^3-y^3-3y\big)$. From the above theorem one gets immediately 
that $\deg(\varphi,M)=-1\cdot\deg(\F,U)=+1$. 
\end{example}

\begin{example}
Let $s=1$, $k=2$, $U=\R^3$ and $g(x_1,x_2,y)=x_1^2-y$, $\varphi(x_1,x_2,y)=(x_1,1+x_2^3,2x_1^2)$. 
Put $\varphi_1(x_1,x_2,y)=(x_1,1+x_2^3)$ and $\varphi_2(x_1,x_2,y)=2x_1^2$. Define 
$\F(x_1,x_2,y)=\big(\varphi_1(x_1,x_2,y),g(x_1,x_2,y)\big)=(x_1,1+x_2^3,x_1^2-y)$. The unique 
zero of $\F$ is $(0,-1,0)$. From the above theorem one gets that 
$\deg(f,M)=-1\cdot\deg(\F,U)=+1$. 
\end{example}

Theorem \ref{formuladeg} and the Additivity Property can be combined to get a formula
for the degree of a tangent vector field tangent valid in a slightly more general 
situation.

\begin{corollary}\label{corg}
Let $U\subseteq\R^k\times\R^s$ be open, $g:U\to\R^s$ a smooth function having $0\in\R^s$ 
as a regular value and let $M=g^{-1}(0)$. Assume $\varphi:M\to\R^k\times\R^s$ is tangent
to $M$ and suppose that there are pairwise disjoint open and connected subsets 
$U_1,\ldots,U_N$ of $U$ such that
\begin{enumerate}
 \item $\varphi^{-1}(0)$ is compact and contained in $\bigcup_{i=1}^NU_i$;
 \item $\partial_2g(x,y)$ is nonsingular for all $(x,y)\in U_i$, $i=1,\ldots,N$.
\end{enumerate}
\begin{equation}\label{fcorg}
  \deg(\varphi,M)=\sum_{i=1}^N \s_i\deg(\F,U_i)
\end{equation}
where $\F:U\to\R^k\times\R^s$ is defined as in Theorem \ref{formuladeg} and $\s_i$ denotes 
the constant sign of $\det\partial_2g(x,y)$ in $U_i$, for $i=1,\ldots,N$.
\end{corollary}

\section{Applications and examples}

This section is devoted to the study of the set of $T$-periodic
solutions of equations \eqref{unoedue}.

Let us introduce some notation.
We shall denote by $C_T(M)$ the set of the continuous $T$-periodic
maps from $\R$ to $M$ with the metric induced by the Banach space
$C_T(\R^m)$ of the continuous $T$-periodic $\R^m$-valued maps
(with the standard supremum norm).
For the sake of simplicity  we make some conventions. We will regard every space 
as its image in the following diagram of natural inclusions 
\[
\begin{array}{cccc}
\left[ 0,\infty \right) \times M & \longrightarrow & \left[ 0,\infty \right)
\times C_T(M) &  \\ 
\uparrow &  & \uparrow &  \\ 
M & \longrightarrow & C_T(M) & 
\end{array}
\]
In particular, we will identify $M$ with its image in $C_T(M)$ under the
embedding which associates to any $\zeta\in M$ the map $\hat \zeta\in C_T(M)$
constantly equal to $\zeta$. Moreover we will regard $M$ as the slice 
$\left\{0\right\} \times M\subset \left[ 0,\infty \right) \times M$ and,
analogously, $C_T(M)$ as $\left\{ 0\right\} \times C_T(M)$. We point out
that the images of the above inclusions are closed.

According to these identifications, if $\Omega$ is an open subset of 
$[0,\infty)\times C_T(M)$, by $\Omega\cap M$ we mean the open subset of $M$ given by all 
$\zeta\in M$ such that the pair $(0,\hat \zeta)$ belongs to $\Omega $. If $\mathcal{O}$ 
is an open subset of $[0,\infty) \times M$, then $\mathcal{O}\cap M$ represents the open 
set $\big\{\zeta\in M:(0,\zeta)\in\mathcal{O}\big\}$.

We say that $(\mu;x,y)\in [0,\infty)\times C_T(M)$ is a 
\emph{solution pair of} \eqref{uno} if $\xi=(x,y)$ satisfies \eqref{uno} for 
$\lambda=\mu$; here the pair $(x,y)$ is thought of as a single element of $C_T(M)$.  
Given $\zeta=(p,q)\in M$, a solution pair 
of the form $(0;\hat p,\hat q)$ is called \emph{trivial}.

\medskip 

Throughout this section $U$ will be an open and connected subset of $\R^k\times\R^s$. We
will always assume that $g:U\to\R^s$ is a smooth function such that $\partial_2g(x,y)$ is 
nonsingular for any $(x,y)\in U$, and $M=g^{-1}(0)$. It will also be convenient, given 
a continuous tangent vector field, $f\colon M\to\R^k\times\R^s$, to denote by $\t f$ an 
arbitrary extension of $f$ to $U$ (as in Remark \ref{remext0}) and to let $\t f_1(x,y)$ be 
the projection of $\t f(x,y)$ on $\R^k$ for any $(x,y)\in U$.

\begin{theorem}\label{tuno}
Let $f:M\to\R^k\times\R^s$, and $h:\R\times M\to\R^k\times\R^s$ be continuous tangent vector 
fields, with $h$ of a given period $T>0$ in the first variable.  Define $\F:U\to\R^k\times\R^s$ 
by $\F(x,y)=\big(\t f_1(x,y),g(x,y)\big)$ for any $(x,y)\in U$. Given an open set 
$\Omega\subseteq[ 0,\infty)\times C_T(M)$, let $\Oc\subseteq\R^m$ be open with the property 
that $\Oc\cap M=\Omega\cap M$. Assume that $\deg (\F,\Oc)$ is well defined and nonzero. Then 
there exists a connected set $\Gamma$ of nontrivial solution pairs for \eqref{uno} in 
$\Omega$ whose closure in $\Omega$ meets $f^{-1}(0)\cap\Omega$ and is not compact. In particular, 
if $M$ is closed and $\Omega=[0,\infty)\times C_T(M)$, then $\Gamma$ is unbounded.
\end{theorem}

\begin{proof}
By Theorem \ref{formuladeg} we have
\[
|\deg(f,\Omega\cap M)|=|\deg(f,\Oc\cap M)|=|\deg(\F,\Oc)|.
\]
Thus, 
$\deg(f,\Omega\cap M)\neq 0$ and the assertion follows from Theorem 3.3 of \cite{FS98}.
\end{proof}

\begin{example}\label{elicafinta}
Let $s=2$, $k=1$, $U=\R^3$ and consider $g:\R\times\R^2\to\R^2$ given by
\[
g(x,y)=g(x;y_1,y_2)=\big(e^{y_1}\cos y_2-x, e^{y_1}\sin y_2+x-1\big).
\]
where $y=(y_1,y_2)$. Clearly, although for each $(x,y)\in\R\times\R^2$
\[
\det\partial_2g(x,y)=\det
\begin{pmatrix}
e^{y_1}\cos y_2 & e^{y_1}\sin y_2 \\
-e^{y_1}\sin y_2 & e^{y_1}\cos y_2
\end{pmatrix}
=e^{2y_1}>0,
\] 
$M=g^{-1}(0)$ is not the graph of a map $x\mapsto y(x)$.
Consider the following ODE on $M$:
\[
\dot \xi=f(\xi),
\]
where $\xi=(x,y_1,y_2)$ and $f$ is the tangent vector field given by
\[
f(x,y_1,y_2)=\big(y_2,
                  y_2(\cos y_2+\sin y_2)e^{-y_1},
                  -y_2 (\cos y_2-\sin y_2)e^{-y_1}\big).
\]
Define $\F(x,y_1,y_2)=\big(y_2,e^{y_1}\cos y_2 -x, e^{y_1}\sin y_2+x-1\big)$, for 
$(x,y_1,y_2)\in\R^3$. From Theorem \ref{formuladeg} we get $\deg(f,M)=\deg(\F,\R^3)=-1$. 

Clearly  $f^{-1}(0)=\{(1,0,0)\}$. Thus, letting $\Omega=[ 0,\infty)\times C_T(M)$ in Theorem 
\ref{tuno}, one has that given any $T$-periodic vector field $h:\R\times\R^3\to\R^3$ tangent to 
$M$ there exists an unbounded connected set $\Gamma $ of nontrivial solution pairs of equation
\[
\dot \xi=f(\xi)+\lambda h(t,\xi),\qquad\lambda\geq 0,
\]
whose closure in $[ 0,\infty)\times C_T(M)$ meets $\{(0,\hat \zeta)\}$ where $\hat \zeta\in C_T(M)$ 
is the function constantly equal to $(1,0,0)$. 
\end{example}

\bigskip
Let us now consider Equation \eqref{due}.
Let $g$ and $h$ be as above, and suppose that $h$ is 
$T $-periodic in the first variable for a given $T>0$. 
We want to derive a continuation result for \eqref{due}, analogous to
Theorem \ref{tuno} above, following \cite{Ca10}.
We say that $(\mu;x,y)\in [0,\infty)\times C_T(M)$ is a 
\emph{solution pair} of \eqref{due} if $\xi=(x,y)$ satisfies \eqref{due} for 
$\lambda=\mu$.  
Given $\zeta=(p,q)\in M$, a solution pair 
of the form $(0;\hat p,\hat q)$ is called \emph{trivial}. 

Define the `average wind' vector field $w^h$ on $M$ by 
\[
 w^h (\xi)=\frac{1}{T} \int_0^T h(t,\xi)dt.
\]
The following result concerns Equation \eqref{due}. 

\begin{theorem}\label{tdue}
Let $h:\R\times M\to\R^k\times\R^s$ be a continuous tangent vector field, of a given period $T>0$ 
in the first variable. Define $\Phi:U\to\R^k\times\R^s$ by $\Phi(x,y)=\big(\t w^h_1(x,y),g(x,y)\big)$ 
for any $(x,y)\in U$. Given an open set $\Omega\subseteq [0,\infty)\times C_T(M)$, let 
$\Oc\subset\R^m$ be an open subset with the property that $\Omega\cap M=\Oc\cap M$. Assume that 
$\deg(\Phi,\Oc)$ is well defined and nonzero. Then there exists a connected set $\Gamma$ of 
nontrivial solution pairs for \eqref{due} in $\Omega$ whose closure in $\Omega$ is not compact and meets the set 
$(w^h)^{-1}(0)\cap\Omega$. In particular, if $M$ is closed and $\Omega=[0,\infty)\times C_T(M)$, then $\Gamma $ is 
unbounded.
\end{theorem}

\begin{proof}
By Theorem \ref{formuladeg} we have
\[
|\deg(w^h,\Omega\cap M)|=|\deg(w^h,\Oc\cap M)|=|\deg(\Phi,\Oc)|.
\]
Thus, 
$\deg(w^h,\Omega\cap M)\neq 0$ and the assertion follows from Theorem 2.2 of \cite{FuPe97}.
\end{proof}

\begin{example}
Let $k=s=1$ and let $U=\R^2$. Consider the map
\[
g(x,y) = y^3+y-x^2.
\]
Clearly, $\partial_2g(x,y)=3y^2+1>0$ for all $\xi=(x,y)\in\R^2$.
Consider the following ODE on $M=g^{-1}(0)$: 
\begin{equation}\label{eqex1}
 \dot \xi=\lambda h(t,\xi),\quad\lambda\geq 0
\end{equation}
where the $2\pi$-periodic tangent vector field $h$ is given by
\[
h(t,x,y)=\left(x+y+\sin t, \frac{2x(x+y+\sin t)}{3y^2+1}\right).
\]
Define
\[
\Phi(x,y)=(x+y,y^3+y-x^2).
\]
Observe that $\Phi^{-1}(0,0)=\{(0,0)\}$ and $\deg(\Phi,\R^2)=1$, so
that Theorem \ref{tdue} applies with $\Omega=[0,\infty)\times C_T(M)$
yielding the existence of an unbounded
branch of solution pairs of \eqref{eqex1}.
\end{example}

\subsection{Applications to a class of Differential-Algebraic Equations (DAEs)} 
Let us now consider applications to semi-explicit
differential-algebraic equations of the form \eqref{DAEs}. 
As above, we will consider the case when 
$U\subseteq\R^k\times\R^s$ is open and connected and $g:U\to\R^s$ is
smooth and 
such that $\partial_2g(x,y)$ is invertible for all $(x,y)\in U$.
For equations \eqref{DAEs} we will write explicitly the tangent vector
fields $f$ and $h$ that carry out the equivalence of
of \eqref{tipo1} with \eqref{uno} and of \eqref{tipo2} with
\eqref{due}, respectively.
The argument is parallel to that of Section \ref{sectan}.

Let us consider equations on an open connected set $U\subseteq\R^k\times\R^s$ of the 
following form:
\begin{equation}\label{dae1}
 \left\{
\begin{array}{l}
 \dot x=F(t,x,y),\\
 g(x,y)=0.
\end{array}
\right.
\end{equation} 
where $F:\R\times U\to\R^k$ is continuous and $g:U\to\R^s$ is smooth and such that 
$\partial_2g(x,y)$ is invertible for all $(x,y)\in U$. 
It is well known (compare \cite[\S 4.5]{KM}) and easy to see that in this situation, 
equation \eqref{dae1} induces a tangent vector field $\Psi$ on $M$, that is, it gives
rise to an ordinary differential equation on $M=g^{-1}(0)\subseteq\R^k\times\R^s$.
In fact, one can see that setting
\[
\Psi(t;x,y) = 
    \left(F(t,x,y)\,,\,-(\partial_2g(x,y))^{-1}\partial_1g(x,y)F(t,x,y)\right),
\]
equation \eqref{dae1} is equivalent to the ordinary differential equation
\[
\dot\xi=\Psi(t,\xi)
\]
on $M$, where $\xi=(x,y)$.

Given continuous maps $\gamma:U\to\R^k$ and $\sigma:\R\times
U\to\R^k$, define
the tangent vector fields 
$f\colon M\to\R^k\times\R^s$ and 
$h\colon M\to\R^k\times\R^s$ on $M$ by 
\begin{gather}
f(x,y)=\big(\gamma(x,y),-(\partial_2g(x,y))^{-1}\partial_1g(x,y)\gamma(x,y)\big),\label{feq}\\
\intertext{and} 
h(t,x,y)=\big(\sigma(t,x,y),-(\partial_2g(x,y))^{-1}\partial_1g(x,y)\sigma(t,x,y)\big).
\label{heq}
\end{gather}
Recalling formula \eqref{phi2}, the above argument shows that
\eqref{tipo1} and \eqref{tipo2} are equivalent to \eqref{uno} 
and \eqref{due}, respectively.
Also, if $\sigma$ is $T$-periodic in the first variable, so is $h$.
\medskip

We are going to use this equivalence to deduce some of the results of \cite{Spa} 
and of \cite{Ca10} for equations of the form \eqref{tipo1} and \eqref{tipo2},
respectively. Let us begin with equations of the form \eqref{tipo1}.

We need to introduce some further notation.
We say that $(\mu;x,y)\in [0,\infty)\times C_T(U)$ is a 
\emph{solution pair of} \eqref{tipo1} if $(x,y)$ satisfies \eqref{tipo1} for 
$\lambda=\mu$.
It is convenient, given any $(p,q)\in\R^k\times\R^s$, to denote by $(\hat p,\hat q)$ 
the map in $C_T(\R^k\times\R^s)$ that is constantly equal to $(p,q)$. A solution pair 
of the form $(0;\hat p,\hat q)$ is called \emph{trivial}. 

Let $\F:U\to\R^k\times\R^s$ be given by $\F(x,y)=\big(\gamma(x,y),g(x,y)\big)$. As 
one immediately checks, $(\hat p,\hat q)$ is a constant solution of \eqref{tipo1} 
corresponding to $\lambda=0$ if and only if $\F(p,q)=(0,0)$. Thus, with this 
notation, the set of trivial solution pairs of \eqref{tipo1} can be written as 
\[
\{(0;\hat p,\hat q)\in [0,\infty)\times C_T(U): \F(p,q)=(0,0)\}.
\] 

Given $\Omega\subseteq[0,\infty)\times C_T(U)$, with $U\cap\Omega$ we denote the set of 
points of $U$ that, regarded as constant functions, lie in $\Omega$. Namely,
\[
  U\cap\Omega=\{(p,q)\in U: (0;\hat p,\hat q)\in\Omega\}.
\]

We are now ready to state and prove a result concerning the $T$-periodic solutions
of  \eqref{tipo1}. 

\begin{theorem}[\cite{Spa}]\label{rami1}
Let $U\subseteq\R^k\times\R^s$ be open and connected. Let $g:U\to\R^s$, $\gamma:U\to\R^k$,
$\sigma:\R\times U\to\R^k$ and $T>0$ be such that $\gamma$ and $\sigma$ are continuous, 
$\sigma$ being $T$-periodic in the first variable, and $g$ is smooth with 
$\partial_2g(x,y)$ invertible for all $(x,y)\in U$. Let also 
$\F(x,y)=\big(\gamma(x,y),g(x,y)\big)$. Given $\Omega\subseteq[0,\infty)\times C_T(U)$ 
open, assume $\deg(\F,U\cap\Omega)$ is well-defined and nonzero. Then, there exists a 
connected set $\Gamma$ of nontrivial solution pairs of \eqref{tipo1} whose closure in 
$\Omega$ is not compact and meets the set $\{(0,\hat p,\hat q)\in\Omega:\F(p,q)=(0,0)\}$.
\end{theorem}

\begin{proof}
Let $f\colon M\to\R^k\times\R^s$ and $h\colon M\to\R^k\times\R^s$ be given by \eqref{feq}
and \eqref{heq}, respectively. Then, as remarked above \eqref{tipo1} is equivalent to 
\eqref{uno} on $M=g^{-1}(0)$. 
This equivalence implies that each pair $(\lambda;x,y)$ can be thought as a solution
pair of \eqref{uno} and vice versa. The assertion follows from Theorem \ref{tuno}.
\end{proof}

The following example could be treated with classical methods because
of the asymptotic behavior at infinity of the implicit function.
Nevertheless, we choose to include it here as an illustration of our results.

\begin{example}
Consider the second order DAE
\begin{equation}\label{molladura}
\left\{
\begin{array}{l}
 \ddot x=-y-\alpha\dot x+\lambda \sigma(t,x,\dot x), \quad \lambda\geq 0\\
 y^3+y-x^5-x=0
\end{array}
\right.
\end{equation}
that represents the motion with friction $-\alpha\dot x$, $\alpha>0$, of a unit mass 
particle constrained to the real axis and attached to the origin with an initially `stiff' 
nonlinear spring (such that the displacement $x$ and the reaction force $-y$ are related 
implicitly by $y^3+y=x^5+x$), and acted on by a $T$-periodic force $\sigma$ depending on 
position and velocity. Let us rewrite equivalently \eqref{molladura} as a first order DAE 
of the form \eqref{tipo1}.
\begin{equation}\label{molla1}
\left\{
\begin{array}{l}
 \dot x_1=x_2\\
 \dot x_2=-y-\alpha x_2+\lambda \sigma(t;x_1, x_2), \quad \lambda\geq 0\\
 y^3+y-x_1^5-x_1=0.
\end{array}
\right.
\end{equation}
Take $U=\R^2\times\R$ and define $\F(x_1,x_2,y)=(x_2,-\alpha x_2-y,y^3+y-x_1^5-x_1)$. 
Since $\deg(\F,U)=1$, Theorem \ref{rami1} yields an unbounded connected set $\Gamma$ of 
nontrivial solution pairs of \eqref{molla1} emanating from the solution constantly equal 
to $(0,0,0)$. 
Clearly, each element of $\Gamma$ corresponds to a nonconstant $T$-periodic solution of 
\eqref{molladura}. In fact, an energy argument shows that \eqref{molladura} has only 
constant periodic solutions for $\lambda=0$. Thus, $\Gamma$ has no intersection with the 
slice $\{0\}\times C_T(U)$. 
\end{example}

In a similar way we deduce a continuation result for equation \eqref{tipo2} from Theorem 
\ref{tdue} above. In the following we will say that $(\mu;x,y)\in [0,\infty)\times C_T(U)$ 
is a \emph{solution pair of} \eqref{tipo2} if $(x,y)$ satisfies \eqref{tipo2} for 
$\lambda=\mu$.
A solution pair 
of the form $(0;\hat p,\hat q)$ will be called \emph{trivial}.

\begin{theorem}[\cite{Ca10}]\label{rami2}
Let $U\subseteq\R^k\times\R^s$ be open and connected. Let $g:U\to\R^s$
be smooth with 
$\partial_2g(x,y)$ invertible for all $(x,y)\in U$, and $\sigma:\R\times U\to\R^k$ continuous 
and $T$-periodic in the first variable. Let also $\Phi:U\to\R^k\times\R^s$ be given by 
$\Phi(x,y)=\big(\Sigma(x,y),g(x,y)\big)$, where
\[
 \Sigma(x,y)=\frac1T \int_0^T \sigma(t,x,y)dt.
\]
Given $\Omega\subseteq[0,\infty)\times C_T(U)$ 
open, assume that $\deg(\Phi,U\cap\Omega)$ is well-defined and nonzero. Then, there exists a 
connected set $\Gamma$ of nontrivial solution pairs of \eqref{tipo2} whose closure in 
$\Omega$ is not compact and meets the set $\{(0,\hat p,\hat q)\in\Omega:\Phi(p,q)=(0,0)\}$.
\end{theorem}

\begin{proof}
Let $h\colon M\to\R^k\times\R^s$ be the tangent vector field on $M$ given by \eqref{heq}.
Then, equation \eqref{tipo2} is equivalent to \eqref{due} on $M=g^{-1}(0)$, and the assertion 
follows from Theorem \ref{tdue}.
\end{proof}

\begin{example}\label{elicavera}
Consider the following DAE in the form \eqref{tipo2} with $T=2\pi$:
\begin{equation}\label{elica1}
 \left\{
\begin{array}{ll}
 \begin{array}{l}
 \dot x_1=\lambda(y_2+\cos t)\\
 \dot x_2=\lambda \big(y_1-2 \cos^2 t \big)
\end{array} & \lambda \geq0\\
\begin{array}{l}
x_1-y_1\cos y_2=0\\
x_2-y_1\sin y_2=0
\end{array} & y_1>0
\end{array}
\right.
\end{equation}
Here, $s=2$, $k=2$, $U=\{(x_1,x_2;y_1,y_2)\in\R^2\times\R^2, y_1>0\}$. Let 
$g:U\to\R^2$ be given by
\[
g(x,y)=g(x_1,x_2;y_1,y_2)=\big(x_1-y_1\cos y_2, x_2-y_1\sin y_2 \big).
\]
where $x=(x_1,x_2)$ and $y=(y_1,y_2)$. One has,
\[
\det\partial_2g(x,y)=\det
\begin{pmatrix}
\cos y_2 & y_1\sin y_2\\
\sin y_2 & -y_1\cos y_2
\end{pmatrix}
=y_1>0.
\]
Clearly, the $2$-dimensional manifold $M$ cannot be written as the graph of a 
function $(x_1,x_2)\mapsto\big(y_1(x_1,x_2),y_2(x_1,x_2)\big)$.
Let $\Phi:U\to\R^4$ be given by
\[
 \Phi(x,y)=\Phi(x_1,x_2;y_1,y_2)=\big(y_2,y_1-1,x_1-y_1\cos y_2,x_2-y_1\sin y_2\big).
\]
A straightforward computation shows that $\Phi^{-1}(0)=\{(1,0,1,0)\}$ and that 
$\deg(\Phi,U)=-1$.
Then, Theorem \ref{rami2}  
applies with $\Omega=[0,\infty)\times C_T(U)$, yielding the existence of
an unbounded branch of nontrivial solution pairs of \eqref{elica1}.
\end{example}

\end{document}